\documentclass[10pt]{amsart}
\usepackage{amsthm, amsmath,cleveref, amsfonts, amssymb, enumerate, textcmds, tensor, enumitem, mathdots}
\usepackage{pst-node}
\usepackage{tikz-cd} 
\usepackage{graphicx,tikz}
\usetikzlibrary{shapes.misc}

\tikzset{cross/.style={cross out, draw=black, minimum size=2*(#1-\pgflinewidth), inner sep=0pt, outer sep=0pt},
cross/.default={1pt}}
\usepackage{enumerate}
\usepackage{pinlabel}
\usepackage[colorinlistoftodos]{todonotes}
\newcommand{\R}{\mathbb R}

\newtheorem{thm}{Theorem}[section]

\newtheorem{prop}[thm]{Proposition}

\DeclareMathOperator{\C}{\mathbb{C}}

\theoremstyle{definition}
\newtheorem{defn}[thm]{Definition}
\newtheorem{remark}[thm]{Remark}

\begin{document}
\title{Non-convexity of extremal length}
\author[Nathaniel Sagman]{Nathaniel Sagman}
\address{Nathaniel Sagman: University of Luxembourg, 2 Av. de l'Universite, 4365 Esch-sur-Alzette, Luxembourg.} \email{nathaniel.sagman@uni.lu}

\begin{abstract}
With respect to every Riemannian metric, the Teichm{\"u}ller metric, and the Thurston metric on Teichm{\"u}ller space, we show that there exist measured foliations on surfaces whose extremal length functions are not convex. The construction uses harmonic maps to $\R$-trees and minimal surfaces in $\R^n.$
\end{abstract}
\maketitle

\begin{section}{Introduction}
Let $\Sigma_g$ be a closed oriented surface of genus $g\geq 2,$ and let $\mathbf{T}_g$ be the Teichm{\"u}ller space of marked Riemann surface structures on $\Sigma_g.$ To any measured foliation $\mathcal{F}$ on $\Sigma_g$ we can associate the extremal length function $\mathbf{EL}_{\mathcal{F}}:\mathbf{T}_g\to (0,\infty).$ Extremal length functions play a large role in Teichm{\"u}ller theory. See, for instance, Kerckhoff's formula \cite[Theorem 4]{Ke} and the Gardiner-Masur compactification \cite{GMe}.

Liu-Su proved that $\mathbf{EL}_{\mathcal{F}}$ is plurisubharmonic, and Miyachi proved the stronger result that it is log-plurisubharmonic (see \cite{LW} and \cite{Mi}). Note that convexity with respect to a Riemannian metric implies plurisubharmonicity. Rafi-Lenhzen proved that, on Teichm{\"u}ller geodesics, extremal length is $K$-quasi-convex, but they also constructed a Teichm{\"u}ller geodesic along which the extremal length is not convex \cite{LR}. Continuing in this direction, Bourque-Rafi proved that the Teichm{\"u}ller metric admits non-convex balls by finding foliations and geodesics where the extremal length is not convex under any reparametrization \cite{BR} (see especially Lemma 1.2 in \cite{BR}).

In this note, we extend the non-convexity result of Rafi-Lenhzen \cite{LR}. Let $\mathcal{C}$ denote the class of (possibly asymmetric) Finsler metrics on $\mathbf{T}_g$ such that for every point $S$ in $\mathbf{T}_g$ and every tangent vector $\mu$ at that point, there is a $C^2$ geodesic starting at $S$ and tangent to $\mu$ at time zero. $\mathcal{C}$ includes every Riemanian metric, notably the Weil-Petersson metric, but also the Teichm{\"u}ller metric and the Thurston metric. Rafi-Lenhzen build an explicit foliation and a Teichm{\"u}ller ray that has pieces along which the slope of the extremal length function decreases. In contrast, we show that convexity fails at an infinitesimal level.
\begin{thm}\label{main}
For all $g\geq 2$ and $m\in \mathcal{C}$, there exists a measured foliation $\mathcal{F}$ on $\Sigma_g$ with real analytic extremal length function and a geodesic $t\mapsto S_t$ for $m$ with the property that $$\frac{d^2}{dt^2}|_{t=0}\mathbf{EL}_{\mathcal{F}}(S_t) <0.$$
In particular, $\mathbf{EL}_{\mathcal{F}}$ is not convex with respect to $m$.
\end{thm}
As noted in \cite{M2}, it follows from the main result of \cite{M2} that with respect to every Riemannian metric on $\mathbf{T}_g$, the energy functional for harmonic maps associated with a Fuchsian representation can be non-convex. By the paper \cite{SS}, the same result holds for (non-Fuchsian) Hitchin representations. We prove Theorem \ref{main} by interpreting extremal length as an energy. Drawing from recent work on minimal surfaces (see \cite{M1},\cite{M2},\cite{MS},\cite{MSS},\cite{SS}), we establish a link between non-convexity of extremal length and instability of minimal surfaces in $\R^n$. 

One of the main takeaways of the proof is that a destabilizing variation of an equivariant minimal surface in $\R^n$ produces a foliation (or even a number of foliations) whose extremal length can be lowered to second order. And although it's probably difficult in practice, if one has the explicit minimal surface data, then one could compute quantities associated with the extremal length (see Remark \ref{end}). 

It would require some care, but one could try to use minimal surfaces to construct a foliation and a geodesic (for some metric) such that, in restriction to the geodesic, the extremal length has a local maximum at time zero. This would imply that the extremal length is not convexoidal for the metric. In fact, it is conjectured in \cite{FMP} that the extremal length systole attains a local maximum at the regular octahedron punctured at its vertices, which would imply that Vorono{\"i}'s criterion fails for the extremal length systole, and moreover that extremal length is not convexoidal for any metric in $\mathcal{C}$ (see \cite[Definition 1.4 and Proposition 1.5]{Bav} for definitions and justification).

Finally, let us remark that a number of questions remain open related to convexity in Teichm{\"u}ller geometry. It is not known if the Teichm{\"u}ller metric convex hull of $3$ points in $\mathbf{T}_g$ can be all of $\mathbf{T}_g$. While sufficiently small Teichm{\"u}ller balls are always convex (the analogous fact holds for any Finsler metric), it is unclear if sets of the form $\{S\in \mathbf{T}_g: \mathbf{EL}_{\mathcal{F}}(S)<\alpha\}$, referred to as horoballs in \cite{BR}, are convex for $\alpha$ small. Our proof of Theorem \ref{main} suggests a new way to probe the convexity question for such horoballs.

\subsection{Acknowledgements}
I'd like to thank Kasra Rafi and Maxime Fortier Bourque for discussion on this topic. I'd also like to thank the anonymous referee for catching some minor errors and sharing helpful comments. I am funded by the FNR grant O20/14766753, \it{Convex Surfaces in Hyperbolic Geometry.}

\end{section}

\begin{section}{Preliminaries}

\subsection{Measured foliations}
Let $\mathcal{S}$ be the set of non-trivial homotopy classes of simple closed curves on $\Sigma_g$, and $\R^{\mathcal{S}}$ the product space with the weak topology. Any $\gamma\in \mathcal{S}$ determines a point in $\R^{\mathcal{S}}$ through the intersection number, 
\begin{equation}\label{int}
    \gamma\mapsto (i(\gamma,\alpha))_{\alpha\in \mathcal{S}}.
\end{equation}
A weighted multicurve is a formal positive linear combination of classes in $\mathcal{S}$. We extend the intersection number to the space of weighted multicurves $\mathcal{WS}$ by $$i\Big (\sum_{j=1}^na_j\gamma_j, \sum_{k=1}^m b_k\alpha_k\Big)= \sum_{j=1}^n \sum_{k=1}^m a_jb_ki(\gamma_j,\alpha_k),$$ which as above yields an embedding from $\mathcal{WS}$ into $\R^{\mathcal{S}}$ via the same map (\ref{int}). To us, the space of measured foliations $\mathcal{MF}$ is the closure of $\mathcal{WS}$ in $\R^{\mathcal{S}}$. Note that the intersection number extends continuously to $\mathcal{MF}$ \cite{Thbook}.

Alternatively, a measured foliation $\mathcal{F}$ is a singular foliation on $\Sigma_g$, the singularities being $k$-prongs, $k\geq 3$, equipped with a transverse measure: an absolutely continuous measure defined on arcs transverse to the foliation and which is invariant under leaf-preserving isotopy. Two measured foliations are measure equivalent if they differ by a leaf-preserving isotopy and Whitehead moves. See \cite[Expos{\'e} 5]{Thbook} for the precise definitions. The intersection function is defined on simple closed curves by integration against the transverse measure.

 Let $S$ be a Riemann surface structure on $\Sigma_g$. The vertical (resp. horizontal) foliation of a holomorphic quadratic differential $\phi$ on $S$ is the singular foliation whose leaves are the integral curves of the line field on $S\backslash \phi^{-1}(0)$ on which $\phi$ is a negative (resp. positive) real number. The singularities are indeed prongs at the zeros, with a zero of order $k$ corresponding to a prong with $k+2$ segments. Both foliations come with transverse measures determined by $|\textrm{Re}\sqrt{\phi}|$ and $|\textrm{Im}\sqrt{\phi}|$ respectively. In this paper, we will always use the vertical foliation.
    
    The Hubbard-Masur theorem asserts that on the given Riemann surface $S$, every measured foliation $\mathcal{F}$ is measure equivalent to one arising from the construction above \cite{HMt}. We refer to the corresponding differential $\phi$ as the Hubbard-Masur differential.

    \subsection{Extremal length (and its regularity)}\label{reg}
    Let $S$ be a Riemann surface structure on $\Sigma_g,$ and $A\subset S$ a doubly connected domain, conformally equivalent to an annulus $\{z\in \C: 1<|z|<R\}.$ The modulus of $A$ is the quantity $$\textrm{Mod}(A) = \frac{1}{2\pi}\log R.$$
    \begin{defn}
The extremal length of a homotopically non-trivial simple closed curve $\gamma$ with respect to $S$ is $$\textrm{EL}(S,\gamma)=\inf_A \frac{1}{\textrm{Mod}(A)},$$ where the infimum is taken over all doubly connected domains $A$ homotopic to $\gamma.$
\end{defn}
     Given a weighted multicurve $\gamma=\sum_{j=1}^na_j\gamma_j,$ let $\mathcal{A}$ be the set of conformally embedded unions of annuli $A=\cup_{i=1}^n A_i\subset S$, with $A_i$ homotopic to $\gamma_i$. We define $$\textrm{EL}(S,\gamma)=\inf_{A\in \mathcal{A}}\sum_{j=1}^n \frac{a_j^2}{\textrm{Mod}(A_i)}$$  (compare with \cite[Definition 3.4, Proposition 3.7]{KPT}). Kerckhoff showed that the map $\textrm{EL}(S,\cdot)$ extends continuously to all measured foliations, defining a map $\textrm{EL}(S,\cdot):\mathcal{MF}\to (0,\infty)$ \cite{Ke}. 

    Fix a measured foliation $\mathcal{F}$ on $\Sigma_g.$ We define the extremal length function on Teichm{\"u}ller space, $\mathbf{EL}_{\mathcal{F}}:\mathbf{T}_g\to (0,\infty)$,  by $$\mathbf{EL}_{\mathcal{F}}(S) = \textrm{EL}(S,\mathcal{F}).$$ 
    In terms of the Hubbard-Masur differential $\phi$, the extremal length is the $L^1$ norm: $$\mathbf{EL}_{\mathcal{F}}(S) = \int_S |\phi|.$$ Recall that the tangent space of $\mathbf{T}_g$ at a surface $S$ identifies with the vector space of harmonic Beltrami forms on $S$. By direct computation, $\mathbf{EL}_{\mathcal{F}}$ is $C^1$, and the derivative is given by $$d(\mathbf{EL}_{\mathcal{F}})_S(\mu)=-4\textrm{Re}\int\phi \mu.$$
    From our understanding, it is unknown if $\mathbf{EL}_{\mathcal{F}}$ is $C^2$, and we'll have to address this point in the main proof. 
    Royden's computation in \cite[Lemma 1]{Ro} seems relevant to this problem, and the papers \cite{Re1} and \cite{Re2} suggest that it is at most $C^2.$ Around a point in $\mathbf{T}_g$ where all zeros of the Hubbard-Masur differential are simple, $\mathbf{EL}_{\mathcal{F}}$ is real analytic (see \cite{Ma}). This condition is generic, and guaranteed when $\mathcal{F}$ has only $3$-pronged singularities and no saddle connections. We do not pursue the general regularity question in the current paper. 
 
    \subsection{Harmonic maps}\label{hmaps}
    We plan to interpret extremal length in terms of harmonic maps to $\R$-trees. As above, let $S$ be a Riemann surface structure on $\Sigma_g$ and $\nu$ a smooth metric that is conformal with respect to the complex structure. Let $(M,d)$ be a complete and non-positively curved (NPC) length space equipped with an action $\rho:\pi_1(\Sigma_g)\to\textrm{Isom}(M,d).$ Let $\tilde{S}$ be the universal cover and $h:\tilde{S}\to (M,d)$ a $\rho$-equivariant and Lipschitz map. Korevaar-Schoen \cite[Theorem 2.3.2]{KS} associate a locally $L^1$ measurable metric $g=g(h)$, defined on pairs of Lipschitz vector fields. If $h$ is a $C^1$ map to a smooth Riemannian manifold $(M,\sigma)$, and the distance $d$ is induced by a Riemannian metric $\sigma$, then $g(h)$ is represented by the pullback metric $h^*\sigma$. Since $\rho$ is acting by isometries, the tensor $g(h)$ descends to $S$. Henceforth we consider it a function on $S$. The energy density is the locally $L^1$ function $$e(h)=\frac{1}{2}\textrm{trace}_\nu g(h).$$
The total energy is $$\mathcal{E}(S,h) = \int_S e(h)dA,$$
where $dA$ is the area form of $\nu$. The measurable $2$-form $e(h)dA$ does not depend on the choice of compatible metric $\nu$, but only on the complex structure. 
\begin{defn}
$h$ is harmonic if it is a critical point for the energy $h\mapsto \mathcal{E}(S,h)$. 
\end{defn}
Let $g_{ij}(h)$ be the components of $g(h)$ in a holomorphic local coordinate $z=x_1+ix_2$. The Hopf differential of $h$ is the measurable tensor on $S$ given in the local coordinate by
\begin{equation}\label{mhopf}
\phi(h)(z)=\frac{1}{4}(g_{11}(h)(z)-g_{22}(h)(z)-2ig_{12}(h)(z))dz^2.
\end{equation}
In the Riemannian setting, (\ref{mhopf}) is 
$$
\phi(h)(z) = h^*\sigma\Big (\frac{\partial}{\partial z},\frac{\partial}{\partial z}\Big )(z)dz^2.$$
When $h$ is harmonic, even in the metric space setting, the Hopf differential is represented by a holomorphic quadratic differential. 

Assume that $\rho$ has the following property: for any Riemann surface $S$ representing a point in $\mathbf{T}_g$, there is a unique $\rho$-equivariant harmonic map $h:\tilde{S}\to (M,d)$. The energy functional on Teichm{\"u}ller space $\mathbf{E}_\rho:\mathbf{T}_g\to [0,\infty)$ is defined by $$\mathbf{E}_\rho(S)=\mathcal{E}(S,h).$$ When $\mathbf{E}_\rho$ is $C^1$ and the associated harmonic map has Hopf differential $\phi$, the derivative in the direction of a harmonic Beltrami form $\mu$ is 
\begin{equation}\label{dformula}
    d(\mathbf{E}_{\rho})_S(\mu)=-4\textrm{Re}\int\phi \mu.
\end{equation}
See \cite{We} for the proof in the metric space context.

\subsection{$\R$-trees dual to foliations}\label{dualtrees}
    \begin{defn}
An $\mathbb{R}$-tree is a length space $(T,d)$ such that any two points are connected by a unique arc, and every arc is a geodesic, isometric to a segment in $\mathbb{R}$.
\end{defn}
Under this definition, $\R$-trees need not be complete. However, every $\R$-tree isometrically embeds into its completion, which itself is an NPC $\R$-tree \cite[Theorem II.1.9]{MSha}. Going forward, we will implicitly embed all $\R$-trees inside their completions and extend all isometries to the completions, in order to discuss equivariant harmonic maps to $\R$-trees.

We concern ourselves with a particular class of actions on $\R$-trees, obtained as follows. Let $\mathcal{F}$ be a measured foliation on $S$ with Hubbard-Masur differential $\phi$. Lifting $\mathcal{F}$ to the universal cover $\tilde{S}$, we define an equivalence relation on $\tilde{S}$ by $x\sim y$ if $x$ and $y$ lie on the same leaf. The quotient space $\tilde{S}/\sim$ is denoted $T$. Pushing the transverse measure down via the projection $\pi: \tilde{S}\to T$ yields a distance function $d$ that turns $(T,d)$ into an $\mathbb{R}$-tree, with an induced action $\rho:\pi_1(\Sigma_g)\to \textrm{Isom}(T,d).$ Under this distance, the projection map $\pi: \tilde{S}\to (T,d)$ is $\rho$-equivariant and harmonic, and the Hopf differential is exactly $\phi/4$ (see \cite[Section 3]{W}). Note that for generic foliations, $(T,d)$ is not complete \cite[Lemma 9.5]{RMG}, so we are indeed using the completion of $(T,d)$. 

The energy density of $\pi$ can be described explicitly: at a point $p\in\tilde{S}$ on which $\phi(p)\neq 0$, $\pi$ locally isometrically factors through a segment in $\mathbb{R}$. In a small neighbourhood around that point, $g(h)$ is represented by the pullback metric of the locally defined map to $\mathbb{R}$. Therefore, it can be computed that
\begin{equation}\label{enho}
    e(\pi)=\nu^{-1}|\phi|/2.
\end{equation}
Similarly, this provides one way to compute $\phi(\pi)=\phi/4$. In view of (\ref{enho}), we will always rescale the metric on $T$ from $(T,d)$ to $(T,2d).$ In this normalization, the total energy is 
\begin{equation}\label{toten}
    \mathcal{E}(S,\pi)=\int_S |\phi|.
\end{equation}
Keeping $\rho$ and varying the source Riemann surface, $\rho$-equivariant harmonic maps always exist and are unique \cite{Wf}, and hence there is an energy functional $\mathbf{E}_\rho$. 
From the formula (\ref{toten}), we deduce the following.
\begin{prop}\label{relator}
Let $\mathcal{F}$ be a measured foliation with Hubbard-Masur differential $\phi$, and $\rho$ the action on the $\R$-tree dual to $\mathcal{F}$. As functions on $\mathbf{T}_g,$ $\mathbf{E}_\rho=\mathbf{EL}_{\mathcal{F}}.$ 
\end{prop}
Accordingly, the same discussion on regularity from Section \ref{reg} applies to $\mathbf{E}_\rho.$ From this point on, we will think about extremal length solely in terms of harmonic maps to $\R$-trees.
\end{section}

\begin{section}{Non-convexity}
Let $m$ be a metric distance function on $\mathbf{T}_g$ in which $(\mathbf{T}_g,m)$ is a length space. We say that a function $F:\mathbf{T}_g\to \R$ is convex with respect to $m$ if for all geodesics $c:[0,1]\to\mathbf{T}_g$, the function $F\circ c:[0,1]\to \R$ is convex. If $c$ is $C^2$ and $F$ is $C^2$ around the image of $c$, then $F\circ c$ is convex if and only if the second derivative is non-negative at all points.

After discussing metrics on $\mathbf{T}_g$ in \ref{metrics}, we recall constructions from \cite{MSS} relating variations of minimal surfaces in $\R^n$ to variations of minimal maps to products of $\R$-trees. We will then use minimal surfaces in $\R^n$ to find a harmonic map to an $\R$-tree (or, a measured foliation) whose energy (extremal length) can be lowered to second order.

    \subsection{Metrics on $\mathbf{T}_g$}\label{metrics}
  Recall the class of metrics $\mathcal{C}$ from the introduction. Let's briefly justify that the Teichm{\"u}ller metric $d_{T}$ and the Thurston metric $d_{Th}$ are contained in $\mathcal{C}.$ This may follow from a general theory of asymmetric Finsler metrics with certain properties, but we couldn't find a source and we prefer to be hands-on.
  \begin{defn}
      The Teichm{\"u}ller metric is defined by $d_T(S,S')=\inf_g \log K(g),$ where $K(g)$ is the maximum quasiconformal dilatation of a quasiconformal map $g:S\to S'$.
  \end{defn}
 Recall that at a Riemann surface $S$, $T_S\mathbf{T}_g$ identifies with the space of harmonic Beltrami forms on $S$. If $\mu$ is any such Beltrami form, away from the finite zero set of $\mu$ we can locally choose a coordinate $z=x+iy$ in which $\mu=d\overline{z}/dz.$ The Teichm{\"u}ller mapping in the direction of $\mu$ at scale $K$ is defined in such a coordinate by 
 \begin{equation}\label{teichmap}
     f_{\mu,K}(x,y) =K^{1/2}x+K^{-1/2}y
 \end{equation}
We define $f_{\mu,K}$ globally by doing (\ref{teichmap}) over the local patches, and extending to all of $S$ by continuity. The Teichm{\"u}ller ray $K\mapsto f_{\mu,K}$ is a geodesic for $d_T$ tangent to $\mu$ at $K=0.$
\begin{defn}
    The Thurston metric is defined by $d_{Th}(S,S')=\inf_g\log \textrm{Lip}(g),$ where $\textrm{Lip}(g)$ is the Lipschitz constant of a Lipschitz map $g$ taking $S$ to $S'$.
\end{defn}
We couldn't find a clean statement in the literature about the existence of Thurston geodesics in a given tangent direction. The result can probably be established through the constructions of Thurston's original paper \cite{Thu}, but one cannot use Thurston's stretch lines from \cite{Thu} directly: by \cite[Theorem 10.5]{Thu}, the set of directions in the unit tangent bundle of $\mathbf{T}_g$ which are tangent to stretch lines have Hausdorff dimension $0$. We'll instead cite the recent work of Pan and Wolf \cite{PW}.

For any Riemann surface $S$ and projective measured lamination on $S$, Pan and Wolf construct a ``harmonic stretch line," which is a Thurston geodesic that in some sense solves an energy-minimization problem. Every unit tangent vector to $\mathbf{T}_g$ at $S$ is tangent to a harmonic stretch line \cite[Remark 1.12]{PW}. Moreover, harmonic stretch lines are special examples of ``piecewise harmonic stretch lines," which are, as stated in Theorem 1.7 of \cite{PW}, real analytic paths in Teichm{\"u}ller space.

\subsection{Harmonic functions}\label{hfunctions}
Let $S$ be a Riemann surface structure on $\Sigma_g$ and $\phi$ a holomorphic quadratic differential. Let $(T,2d)$ be the dual $\R$-tree with action $\rho$ and harmonic projection map $\pi:\tilde{S}\to (T,2d)$. Assume that $\phi$ is the square of an abelian differential $\alpha$. The cohomology class of the harmonic $1$-form $\textrm{Re}\alpha$ determines a representation $\chi:\pi_1(\Sigma_g)\to (\mathbb{R},+),$ and integrating from a basepoint $p\in \tilde{S}$ yields a $\chi$-equivariant harmonic function $$h:\tilde{S}\to \R, \hspace{1mm} f_i(z) = \int_p^z \textrm{Re}\alpha.$$ 
We can compute directly that $\phi(h)=\phi$ and that for any choice of conformal metric on $S$, $e(h)=e(\pi)$.
Geometrically, $h$ is related to $\pi$ through the folding map $p$, which is a map $p:(T,d)\to \R$ satisfying $h=p\circ \pi$ and restricting to an isometry on geodesic segments of $(T,d)$ (see \cite[Section 4]{MS}). For any Riemann surface, $\chi$-equivariant harmonic functions exist, but they are unique only up to translations of $\R^n$. Nevertheless, the energy density is independent of the choice of harmonic function, and it is possible to choose the harmonic functions locally to vary real analytically with the choice of Riemann surface (see \cite[Section 5]{EL}). Thus, as in \ref{hmaps}, we may define a (real analytic) energy functional $\mathbf{E}_\chi$ on $\mathbf{T}_g.$

In general, there is a degree $2$ branched covering $\tau: C\to S$ in which $\phi$ lifts to a square $\tilde{\phi}=\alpha^2$, disconnected if $\phi$ is already a square, and which has the universal property that any other branched cover on which $\phi$ lifts to a square must factor through $\tau$. We can repeat the above construction on $C,$ obtaining an equivariant harmonic map to an $\R$-tree $\pi':\tilde{C}\to (T',2d')$ that folds onto an equivariant harmonic function $h$ from $\tilde{C}$ to $\R$. $C$ comes with a holomorphic involution that negates $\alpha$. The involution leaves the energy densities of $\pi'$ and $f$ invariant, so that they descend all the way to $S$, where they agree with that of the map to the original $\R$-tree.
   
     \subsection{Minimal maps}\label{minimal}
     Classically, a minimal map to a Riemannian manifold is a harmonic and conformal immersion. An immersion is conformal precisely when the Hopf differential vanishes identically. For an NPC space $(M,d)$, we make the following definition.
    \begin{defn}
        $h:\tilde{S}\to (M,d)$ is minimal if it is harmonic and $\phi(h)=0.$
    \end{defn}
In the presence of a $C^1$ energy functional $\mathbf{E}_\rho,$ by (\ref{dformula}), $S$ is minimal if and only if $S$ is a critical point of $\mathbf{E}_{\rho}.$ For equivariant maps to $\R^n$, minimal maps are also critical points of the area functional $$f\mapsto A(f) = \int_{\Sigma_g} dA_f,$$ where $dA_f$ is the area form of the pullback of the Euclidean metric by $f$. We record the following consequence of the definitions. Let $(X,d)$ be a product of NPC spaces $(M_i,d_i).$ 
\begin{prop}\label{product}
    $h:\tilde{S}\to (X,d)$ is harmonic if and only if every component map $h_i:\tilde{S}\to (M_i,d_i)$ is harmonic. Moreover, $\phi(h)=\sum_{i=1}^n \phi(h_i).$
\end{prop}
    Let $\phi_1,\dots, \phi_n$ be holomorphic quadratic differentials on $S$ and let $(M,d)$ be the product of the dual $\R$-trees, with product action $\rho$ and product of projection maps $\pi=(\pi_1,\dots, \pi_n):\tilde{S}\to (M,d).$ The energy functional $\mathbf{E}_\rho$ is the sum of the component energy functionals. Similar to the previous subsection, there is a degree $2^n$ branched covering $\tau: C\to S$ on which each $\phi_i$ lifts to a square, with the analogous universal property, and which comes with $n$ commuting holomorphic involutions that each negate a $1$-form. There is a product representation $\chi=(\chi_1,\dots, \chi_n):\pi_1(C)\to (\mathbb{R}^n,+)$ with energy functional $\mathbf{E}_\chi$ and equivariant harmonic function $h=(h_1,\dots, h_n):\tilde{C}\to \R^n$. Proposition \ref{product} gives the observation below.
    \begin{prop}
    $h$ is minimal if and only if $\pi$ is minimal if and only if $\sum_{i=1}^n \phi_i=0.$
    \end{prop}
    Let us thus assume the $\phi_i$'s sum to zero. The main input toward Theorem \ref{main} is Proposition \ref{selfmapprop} below, which is used to turn variations of $h$ into variations of $\pi$. We will need to restrict to a class of variations. Set $\textrm{Var}_\tau(h)$ to be the space of $C^\infty$ functions $\dot{h}:\tilde{C}\to \R^n$ that are invariant under $\pi_1(C)$ and the lifts to $\tilde{C}$ of the $n$ holomorphic involutions. Of course, such $\dot{h}$ is equivalent to a function on $S$. 
\begin{prop}[Propositions 4.4, 4.6, and 5.1 of \cite{MSS}]\label{selfmapprop}
        Let $\dot{h}\in \textrm{Var}_\tau(h)$. For every $\epsilon>0$, there exists a $C^\infty$ path of Riemann surfaces $t\mapsto C_t$ and $C^\infty$ paths of $C^\infty$ maps $t\mapsto f_i^t:C\to C_t$ starting at the identity such that 
        \begin{equation}\label{selfmap}
            \frac{d^2}{dt^2}|_{t=0}\sum_{i=1}^n \mathbf{E}_{\chi_i}(C_t, h_i\circ (\tilde{f}_i^t)^{-1})\leq \frac{d^2}{dt^2}|_{t=0}A(h_t) +\epsilon,
        \end{equation}
        where $\tilde{f}_i^t$ is the lift to $\tilde{C}$. The Riemann surfaces $C_t$ descend through the branched cover $\tau$ to Riemann surface structures $S_t$ on $S$. Similarly, the $f_i^t$'s descend to $f_i^t:S\to S_t.$
    \end{prop}
    Very briefly: without perturbing the area too much, one can modify $\dot{h}$ to be zero in a neighbourhood of the zeros of the $\phi_i$'s. There is a canonical way to pull such a variation back to $n$ vector fields on the surface $\tilde{C}$, which then generate flows $t\mapsto f_i^t$. With respect to the conformal structure of the image of $h_t=(h_1\circ (f_1^t)^{-1},\dots, h_n\circ (f_n^t)^{-1})$, which we label $C_t,$ the energy of $h_t$ is equal to its area. By invariance properties of $\dot{h},$ everything can be chosen to descend to $\Sigma_g$. Such a self-maps variation gives a variation of $\pi$,
    \begin{equation}\label{pit}
        \pi_t=(\pi_1 \circ (f_1^t)^{-1},\dots, \pi_n\circ (f_n^t)^{-1}).
    \end{equation}
If $\tilde{f}_i^t$ also denotes the lift to $\tilde{S}$, then we can see by the local isometric factoring described in \ref{dualtrees}, or the folding map of \ref{hfunctions},
\begin{equation}\label{eneq}
    e(\pi_i\circ (\tilde{f}_i^t)^{-1})=e(h_i\circ (\tilde{f}_i^t)^{-1})
\end{equation}
(note that both densities descend to $S$, which is where the equation (\ref{eneq}) is defined). 

Finally, we will want destabilizing variations.
    \begin{defn}
         $h$ is $\tau$-unstable if there exists $\dot{h}\in \textrm{Var}_{\tau}(h)$ such that $$\frac{d^2}{dt^2}|_{t=0}A(h+t\dot{h})<0.$$
    \end{defn}
For all $g\geq 3$ and $n\geq 3$, we can take $C=S$: one can find abelian differentials on $S$ whose squares sum to $0$ that give an unstable minimal map. One way to produce such a map is to lift an unstable minimal surface in the $3$-torus to the universal covers. For $g=2$, it turns out that any equivariant minimal surface is stable. However, we proved
  \begin{thm}[Section 5.3 in \cite{MSS}]\label{g2}
        There exists a Riemann surface $S$ of genus $2$ with $\phi_1,\dots, \phi_n,$ $n\geq 3,$ summing to $0$ and which give a non-trivial branched cover $\tau:C\to S$ and a $\tau$-unstable minimal map $h:\tilde{C}\to \R^n.$
    \end{thm}
    \subsection{Proof of Theorem \ref{main}}
Resuming the setup from Section \ref{minimal}, assume that $h$ is destabilized by some $\dot{h}\in \textrm{Var}_\tau(h)$. Choose $\epsilon>0$ small enough so that $$\frac{d^2}{dt^2}|_{t=0}A(h+t\dot{h})+\epsilon<0.$$ Applying Proposition \ref{selfmapprop}, defining $\pi_t$ as in (\ref{pit}) and using (\ref{eneq}), we arrive at the following. 
\begin{prop}\label{startingpoint}
    There exist $C^\infty$ paths of $C^\infty$ maps to Riemann surfaces $t\mapsto f_i^t:S\to S^t$ starting at the identity such that 
    \begin{equation}\label{smunstable}
        \frac{d^2}{dt^2}|_{t=0}\sum_{i=1}^n\mathcal{E}(S_t,\pi_i\circ (f_i^t)^{-1})\leq 2^{-n}\frac{d^2}{dt^2}|_{t=0}A(h+t\dot{h}) +\epsilon<0.
    \end{equation}
\end{prop}
    All recorded examples of $\tau$-unstable maps, in particular examples from Theorem \ref{g2}, come from differentials with even order zeros. Recalling our regularity concerns from Section \ref{reg}, we need the Proposition below. Say that a holomorphic quadratic differential $\phi$ is generic if it has only simple zeros. Generic quadratic differentials on $S$ form an open and dense subset. 
\begin{prop}\label{perturb}
    For all $g\geq 2,$ $n\geq 3,$ we can choose generic holomorphic quadratic differentials $\phi_i$ that give maps to $\R$-trees $\pi_i,$ and $C^\infty$ paths $t\mapsto f_i^t:S\to S_t$ starting at the identity such that (\ref{smunstable}) holds: $$\frac{d^2}{dt^2}|_{t=0}\sum_{i=1}^n\mathcal{E}(S_t,\pi_i\circ (f_i^t)^{-1})<0.$$
\end{prop}
Morally, we're using that instability in the sense of (\ref{smunstable}) is an open property. To formalize the argument, we borrow a formula from Reich-Strebel.
\begin{prop}[Equation 1.1 in \cite{RS} and Proposition 3.1 in \cite{MSS}]
    Let $\pi:\tilde{S}\to (T,2d)$ be any equivariant map to an $\R$-tree with Hopf differential $\phi$ and let $f:S\to S'$ be any quasiconformal map to another Riemann surface $S'$ with lift $\tilde{f}$ to $\tilde{S}$ and Beltrami form $\mu$.
    \begin{equation}\label{RSforr3}
    \mathcal{E}(S,\pi\circ \tilde{f}^{-1}) -\mathcal{E}(S,\pi) =  -4\textrm{Re} \int_{S} \phi\cdot \frac{ \mu}{1-|\mu|^2}dxdy + 4\int_{S} |\phi_i|\cdot \frac{|\mu|^2}{1-|\mu|^2}dxdy.  
\end{equation}
\end{prop} 

\begin{proof}[Proof of Proposition \ref{perturb}]
Begin with the data $\phi_i,f_i^t$ from Proposition \ref{startingpoint}. The $\phi_i$'s may not be generic, but we know that (\ref{smunstable}) holds. Let $\mu_i^t$ be the Beltrami form of $f_i^t$, and $\alpha_i$ the $C^\infty$ $(1,-1)$-form and $\beta_i$ the $C^\infty$ function on $S$ described by $$\alpha_i(z) = \frac{d^2}{dt^2}\bigg|_{t=0} \frac{ \mu_i^t(z)}{1-|\mu_i^t(z)|^2}, \hspace{1mm} \beta_i(z)= \frac{d^2}{dt^2}\bigg|_{t=0} \frac{|\mu_i^t(z)|^2}{1-|\mu_i^t(z)|^2}.$$
 By the formula (\ref{RSforr3}),
\begin{align*}
\frac{d^2}{dt^2}\bigg|_{t=0} \sum_{i=1}^n \mathcal{E}(S_t,\pi_i\circ (f_i^t)^{-1}) &= \frac{d^2}{dt^2}\bigg|_{t=0} \sum_{i=1}^n \Big (\mathcal{E}(S_t,\pi_i\circ (f_i^t)^{-1}) - \mathcal{E}(S,\pi_i)\Big ) \\
    &=\frac{d^2}{dt^2}\bigg|_{t=0} 4\sum_{i=1}^n\Big (-\textrm{Re} \int_{S} \phi_i\cdot \frac{ \mu_i^t}{1-|\mu_i^t|^2}dxdy + \int_{S} |\phi_i|\cdot \frac{|\mu_i^t|^2}{1-|\mu_i^t|^2}dxdy\Big ).
\end{align*}
Taking the derivative into the integral,
\begin{equation}\label{2var}
    0>\frac{d^2}{dt^2}\bigg|_{t=0} \sum_{i=1}^n \mathcal{E}(S_t,\pi_i\circ (f_i^t)^{-1})=4\textrm{Re} \sum_{i=1}^n\Big (-\int_S \phi_i\cdot\alpha_i + \int_S |\phi_i|\cdot \beta_i dA\Big ).
\end{equation}
We perturb $\phi_1,\dots, \phi_{n-1}$ ever so slightly to be generic, and then redefine $$\phi_n:=-\phi_1-\dots - \phi_{n-1},$$ which is, of course, very close to the original $\phi_n$. If $\phi_n$ is not generic, then we perturb it to be. By openness, if our perturbation of $\phi_n$ is sufficiently small, then if we redefine $\phi_{n-1}$ so that the sum of the $\phi_i$'s is again zero, $\phi_{n-1}$ will still be generic. Since $\alpha_i$ and $\beta_i$ are uniformly bounded, it is clear from the right hand side of (\ref{2var}) that if the perturbations are small enough, then both sides of (\ref{2var}) remain negative. So, we can take these new $\phi_i$'s to be our holomorphic quadratic differentials, and keep the same Riemann surfaces $S_t$ and paths of $C^\infty$ maps $t\mapsto f_i^t:S\to S_t.$
\end{proof}
\begin{proof}[Proof of Theorem \ref{main}]
    We find an $\R$-tree such that energy can be decreased to second order. We will then invoke Proposition \ref{relator} to say that the same happens for the extremal length of the associated foliation.
    
    Our starting point is Theorem \ref{g2}: we fix a $\tau$-unstable minimal map from a branched cover of a Riemann surface $S$ of genus $g$. By Proposition \ref{perturb}, we can adjust the Hopf differentials of the component maps to obtain generic quadratic differentials $\phi_1,\dots, \phi_n$ that yield an action $\rho=(\rho_1,\dots, \rho_n)$ on a product of $\R$-trees and an equivariant minimal map $\pi=(\pi_1,\dots, \pi_n)$, as well as paths of $C^\infty$ maps $t\mapsto f_1^t,\dots, f_n^t:S\to S_t$ starting at the identity such that
    \begin{equation}\label{selfmaps}
        \frac{d^2}{dt^2}\bigg|_{t=0}\sum_{i=1}^n\mathcal{E}(S_t,\pi_i\circ (f_i^t)^{-1})<0.
    \end{equation}
Since the $\phi_i$'s are generic, their energy functionals on $\mathbf{T}_g$ are all real analytic. By the definition of minimality and (\ref{dformula}), 
    \begin{equation}\label{firstd}
        \frac{d}{dt}|_{t=0}\mathbf{E}_\rho(S_t)=0 \textrm{ and } \frac{d}{dt}\bigg|_{t=0}\sum_{i=1}^n\mathcal{E}(S_t,\pi_i\circ (f_i^t)^{-1})=0.
    \end{equation}
Since harmonic maps minimize energy, for all $t$, 
\begin{equation}\label{decreasing}
    \mathbf{E}_\rho(S_t)=\sum_{i=1}^n \mathbf{E}_{\rho_i}(S_t)\leq \sum_{i=1}^n\mathcal{E}(S_t,\pi_i\circ (f_i^t)^{-1}).
\end{equation}
It follows from (\ref{selfmaps}), (\ref{firstd}), and (\ref{decreasing}) that 
\begin{equation}\label{secondd}
    \frac{d^2}{dt^2}|_{t=0}\mathbf{E}_{\rho}(S_t)\leq \frac{d^2}{dt^2}\bigg|_{t=0}\sum_{i=1}^n \mathcal{E}(S_t,\pi_i\circ (f_i^t)^{-1})<0.
\end{equation}
    By the first equation in (\ref{firstd}), the left hand side of (\ref{secondd}) does not depend on the specific path in Teichm{\"u}ller space, but only on the initial tangent vector $\mu$. Thus, if we fix a metric $m$ in $\mathcal{C}$, we can replace the path with the geodesic for $m$ starting at $S$ and tangent to $\mu$ at time zero, say $t\mapsto S_t'$. Finally, since the energy splits into the energies of the component maps, for (\ref{secondd}) to hold along our path there must be at least one component representation $\rho_i$ such that $$\frac{d^2}{dt^2}|_{t=0}\mathbf{E}_{\rho_i}(S_t')<0.$$ If $\mathcal{F}$ is the measured foliation corresponding to $\rho_i,$ then by Proposition \ref{relator}, $$\frac{d^2}{dt^2}|_{t=0}\mathbf{EL}_{\mathcal{F}}(S_t')<0.$$ This completes the proof.
\end{proof}

\begin{remark}
   One could instead ask about convexity with respect to a connection: a function is convex with respect to a connection if the restriction to every geodesic for that connection is a convex function. This setting is considered in \cite{Bav}. Our proof shows that extremal length is not convex for any connection that admits $C^2$ geodesics through every tangent vector of $\mathbf{T}_g.$
\end{remark}

\begin{remark}\label{end}
    In principle, one can write down the tangent vectors explicitly. One begins with a $\tau$-unstable minimal surface and a destabilizing variation. For example, one could work in genus $3$, and take $C=S$ and any non-planar equivariant minimal map from $\tilde{S}\to \R^3$ with its destabilizing unit normal variation (see \cite[Section 5.3]{MSS}). The proofs of Propositions 4.4, 4.6, and 5.1 in \cite{MSS} explain how to build the flows $f_1^t,f_2^t,f_3^t.$ One can then try to compute the corresponding path $t\mapsto S_t$ and take the derivative at time zero, although the computation may be involved. 
\end{remark}
\end{section}

\bibliographystyle{plain}
\bibliography{bibliography}

\end{document}